\documentclass[a4paper,11pt]{article}
\usepackage{fancyhdr}
\usepackage{amsmath}
\usepackage{cases}
\usepackage{cite}
\usepackage{mathrsfs}
\usepackage{amsfonts}
\usepackage{xcolor}
\usepackage{amsmath,amsthm,amssymb}
\usepackage[colorlinks,citecolor=green]{hyperref}
\numberwithin{equation}{section}
\usepackage{indentfirst}
\usepackage{stmaryrd}
\usepackage{enumitem}

\newtheorem{thm}{Theorem}[section]
\newtheorem{lemma}[thm]{Lemma}
\newtheorem{prop}[thm]{Proposition}

\newtheorem{remark}[thm]{Remark}

\renewcommand{\thefootnote}{\fnsymbol{footnote}}

\numberwithin{equation}{section}
\begin{document}
\title
{
\large{\textbf{Enhanced Dissipation and Global Well-Posedness for a Three-Dimensional Flame Propagation Model with Couette Flow}}}
\author
{Yoshiyuki Kagei$^{a}$, \ \ Lijuan Wang$^{b, {\ast}}$ \\
\footnotesize{\it $^{a}$ Department of Mathematics, Institute of Science Tokyo, Tokyo, 152-8551, Japan}\\
\footnotesize{\it $^{b}$ Shanghai University of International Business and Economics, Shanghai, 201620,  China}}
\def\thefootnote{\fnsymbol{footnote}}
\footnotetext[1]{ Corresponding author.}
\date{}
\maketitle
\noindent\textbf{Abstract:}
We study a three-dimensional gravity-induced flame front model under a Couette flow.
By exploiting the enhanced dissipation induced by the Couette flow, we prove
global-in-time well-posedness of the Cauchy problem in $\mathbb{R}^3$ and derive decay
estimates for the solution and its spatial derivatives in $L^p$ norms for all $p \ge 1$.
The analysis is based on a Green's function approach for the associated
variable-coefficient linearized operator.
Since an explicit representation of the Green's function is unavailable, we first
establish decay estimates in the spectral domain and then transfer them to physical
space.
These results show that enhanced dissipation induced by the Couette flow is the key
mechanism leading to global existence in the whole space, in the large initial data
regime.
\normalsize
\section{Introduction}
Models of flame front propagation have attracted considerable attention in both the combustion and applied mathematics communities.
The classical work of Landau \cite{Landau1944} first revealed the hydrodynamic instability mechanism of premixed flames, laying the foundation for subsequent studies of flame front dynamics.

Based on this instability theory, Sivashinsky and his collaborators derived nonlinear evolution equations describing unstable flame fronts, among which the Kuramoto-Sivashinsky equation (KSE) and its variants play a central role \cite{Sivashinsky1977, MichelsonSivashinsky1977, Sivashinsky1983}.
These models capture the competition between destabilizing mechanisms, such as hydrodynamic and gravitational effects, and stabilizing mechanisms provided by diffusion and higher-order dissipation.

In this article, we study a particular flame propagation model derived by Margolis and Sivashinsky \cite{Sivashinsky1984}, namely, the gravity-induced flame propagation model
\begin{equation}\label{eq:flame}
\left\{
\begin{array}{ll}
\partial_t \phi + \Delta \phi + 4\Delta^2 \phi + \kappa \phi
= -\dfrac{1}{2}|\nabla \phi|^2,\\[4pt]
\phi(0,x,y,z)=\phi_0(x,y,z),
\end{array}
\right.
\end{equation}
where $\phi(x,y,z,t)$ denotes the location of the flame front, $\kappa>0$ represents the strength of gravity, and throughout this paper we consider the whole-space setting $(x,y,z)\in\mathbb{R}^3$.
This model was first proposed as a weakly nonlinear description of flame front dynamics under gravitational effects and can be viewed as a higher-dimensional extension of classical Kuramoto-Sivashinsky-type equations.

From a mathematical perspective, equation \eqref{eq:flame} and, more generally, Kuramoto-Sivashinsky-type equations are difficult to analyze in dimensions greater than one due to the absence of a maximum principle and the presence of competing destabilizing and stabilizing mechanisms.
These features pose substantial analytical challenges, and as a consequence, to the best of our knowledge, relatively few rigorous results are available for such models in higher dimensions.

Taking the Kuramoto-Sivashinsky equation as a representative example, the one-dimensional theory is by now well developed \cite{bro, ColletEckmann1990, gia, gold, gol, ott}.
In contrast, in higher dimensions the analysis remains much less understood, largely because of the lack of a maximum principle.
In particular, global well-posedness in two dimensions is known only under restrictive assumptions, such as for thin domains and for the anisotropically reduced Kuramoto-Sivashinsky equation \cite{ben,la,se}, without growing modes \cite{ambrose-1} and \cite{feng2021}, or with only one growing mode in each direction \cite{ambrose-2}, for small data.

These analytical difficulties naturally lead to the question of how one can stabilize the flame front dynamics.
A particularly effective idea is to enhance dissipation by introducing suitable advective flows.
From a mathematical point of view, adding a flow field to strengthen dissipation in parabolic equations is a novel and powerful strategy.
In 2008, Constantin, Kiselev, Ryzhik and Zlat$\check{o}$s~\cite{ckrz} studied parabolic equations with a designed advective terms and introduced the concept of relaxation-enhancing flows. Since then, this mechanism has been widely used to prevent blow-up in chemotaxis models and has generated a number of interesting results \cite{6-kx, hf}.

Another simple and effective way to enhance dissipation is to introduce the Couette flow.
Compared with relaxation-enhancing flows, Couette flows have a much simpler and more explicit structure.
Nevertheless, they can still produce strong stabilizing effects.
The use of Couette flows to suppress blow-up in chemotaxis models has been extensively studied, and many important results have been obtained \cite{bedrossian,h,fsw,S-WWK,zzz,h1,Ding}.

For the Kuramoto-Sivashinsky equation, Feng and Mazzucato~\cite{feng2021} proved global existence for the advective Kuramoto-Sivashinsky equation on the two-dimensional torus under relaxation-enhancing flows, while Zelati~\cite{Zelati2021} established global well-posedness in the presence of a steady shear flow.

In this article, we study the three-dimensional gravity-induced flame propagation model with a Couette flow $(Ay, 0, 0)$,
\begin{equation}\label{eq:mainequation}
\left\{
\begin{array}{ll}
\partial_t \phi + Ay\partial_x \phi+\Delta\phi+ 4\Delta^2 \phi + \kappa \phi= -\dfrac{1}{2}|\nabla \phi|^2,
\\[5pt]
\phi(0,x,y,z)=\phi_0(x,y,z), \quad (x,y,z)\in\mathbb{R}^3,
\end{array}
\right.
\end{equation}
where $A>0$ denotes the amplitude of the Couette flow.

In this paper, we show that the enhanced dissipation induced by a sufficiently large Couette flow can establish global existence for the Cauchy problem. Unlike most existing studies conducted in periodic domains, we consider the problem in the whole space $\mathbb{R}^3$, where the Fourier spectrum is continuous and mode decomposition is no longer available. Moreover, the variable coefficient $Ay$ in the linearized operator introduces additional analytical challenges.

Our analysis relies on a nonstandard use of the Green's function method. Constructing Green's functions for linearized equations with variable coefficients is highly nontrivial and depends sensitively on the structure of the coefficients.
Over the past two decades, this approach, often combined with microlocal and harmonic analysis, has been successfully applied to study the existence and long-time behavior of solutions to nonlinear equations \cite{liu-1, liu-2, liu-3, liu-4, liu-5, yu}.
For example, Deng, Shi and Wang~\cite{deng} showed that a Gaussian kernel associated with a Couette flow can prevent blow-up through the action of a sufficiently large Couette flow. In the present work, we demonstrate that the variable coefficient $Ay$ plays a key role in dissipation enhancement, and that its presence provides a stabilizing mechanism rather than merely a technical difficulty.

In general, it is challenging to derive explicit expressions for the Green's functions of such variable-coefficient equations.
Indeed, obtaining a closed-form expression for the Green's function of equation \eqref{eq:mainequation} remains intractable.
To overcome this difficulty, we derive pointwise estimates for the Fourier transform of the Green's function in spectral space and then transfer these estimates back to physical space.

The main result of this paper is stated in the following theorem.

\begin{thm}\label{thm:main}
Let $\phi_0(x,y,z) \in W^{4,\infty}(\mathbb{R}^3)\cap L^1(\mathbb{R}^3)$ be the initial data, and let $\varepsilon>0$ be a fixed constant, with the assumption that
$
\kappa > \frac{1}{16} + \varepsilon.
$
Then there exists a positive constant $A_0 = A_0(\phi_0)$ such that, for any $A \ge A_0$, the Cauchy problem \eqref{eq:mainequation} admits a unique global classical solution
\[
\phi(t,x,y,z) \in C\big([0,\infty);\, W^{4,\infty}(\mathbb{R}^3)\cap L^1(\mathbb{R}^3)\big).
\]

Moreover, the solution satisfies the following decay estimates for all derivatives:
\begin{equation}\label{eq:decay}
\| D^k \phi(t) \|_{L^p}
\le C\, e^{-\varepsilon t} \,(1+t)^{-\frac{3}{4}(1-\frac{1}{p})-\frac{k}{4}} \,(1+(At)^4)^{-\frac{1}{4}(1-\frac{1}{p})},
\qquad 0 \le k < 4, \ \ p \ge 1,
\end{equation}
where $D^k = \partial_x^{k_1}\partial_y^{k_2}\partial_z^{k_3}$ denotes any multi-index derivative with $k_1+k_2+k_3=k$, and the constant $C>0$ depends only on the initial data $\phi_0$ and is independent of $t$.
\end{thm}

We conclude this section with several remarks that emphasize the significance of Theorem~\ref{thm:main} and the role of the Couette flow in the global existence and decay properties.
\begin{remark}
\begin{itemize}
  \item []
\end{itemize}
\noindent\textnormal{\textbf{(1)}}~
When $A=0$, the background Couette flow is absent. In this case, although global existence
of classical solutions to \eqref{eq:mainequation} in the whole space is known for sufficiently
small initial data, the corresponding large-data problem remains a challenging open problem
in three dimensions.

\noindent\textnormal{\textbf{(2)}}~
Theorem~\ref{thm:main} demonstrates that the enhanced dissipation induced by the Couette
flow is sufficient to guarantee global-in-time existence of classical solutions to the
three-dimensional gravity-induced flame front model.

\noindent\textnormal{\textbf{(3)}}~
Compared with the case $A=0$, the decay estimate \eqref{eq:decay} includes an additional
factor $(1+(At)^4)^{-\frac{1}{4}\left(1-\frac{1}{p}\right)}$, which reflects the
enhanced dissipation mechanism induced by the Couette flow.
\end{remark}

The remainder of this paper is organized as follows. In Section~2, we
present the Green's function and establish the associated
\texorpdfstring{$L^p$}{Lp} estimates. Section~3 is devoted to the local
and global existence of solutions. In Section~4, we derive decay
estimates for the solutions.

Throughout this paper, $C$ denotes a generic constant that may change
from line to line.

\section{Green's Function and Computational Lemmas}
\subsection{Green's Function}
To study solutions of \eqref{eq:mainequation}, we first derive the Green's function $\mathbb{G}(t,x,y,z)$ (i.e.\ the fundamental solution) associated with its linearized equation.
\begin{equation}\label{eq:green}
\left\{
\begin{array}{ll}
&\partial_t \mathbb{G}+Ay\partial_x \mathbb{G}+\Delta \mathbb{G}+4\Delta^2 \mathbb{G}+\kappa \mathbb{G}=0,
\\[4pt]
&\mathbb{G}(0,x-x', y-y', z-z')=\delta(x-x', y-y', z-z').
\end{array}
\right.
\end{equation}
Taking the Fourier transform of \eqref{eq:green} with respect to the spatial variables $(x, y, z)$ yields
\begin{equation}\label{eq:green-1}
\left\{
\begin{array}{ll}
\partial_t \widehat{\mathbb{G}}-A\xi\partial_\eta \widehat{\mathbb{G}}
-(\xi^2+\eta^2+\zeta^2)\widehat{\mathbb{G}}
+4(\xi^2+\eta^2+\zeta^2)^2\widehat{\mathbb{G}}
+\kappa\widehat{\mathbb{G}}=0,
\\[2mm]
\widehat{\mathbb{G}}(0,\xi,\eta,\zeta)
=\exp\big(-ix'\xi-iy'\eta-iz'\zeta\big).
\end{array}
\right.
\end{equation}

A direct computation shows that the Fourier transform of the Green's function admits the factorization
\[
\widehat{\mathbb{G}}(t,\xi,\eta,\zeta; x^\prime, y^\prime, z^\prime)
=\widehat{\mathbb{G}}_1(t,\xi,\eta,\zeta; x^\prime, y^\prime, z^\prime)\,\cdot
\,\widehat{\mathbb{G}}_2(t,\xi,\eta,\zeta),
\]
where
\[
\widehat{\mathbb{G}}_1(t,\xi,\eta,\zeta; x^\prime, y^\prime, z^\prime)
=\exp\big(-i x'\xi-i y^\prime(\eta+A\xi t)  -i z'\zeta\big),
\]
and
\[
\widehat{\mathbb{G}}_2(t,\xi,\eta,\zeta)
=\exp\Big(\int_0^t\big[\big(\xi^2+(\eta+A\xi s)^2+\zeta^2\big)-\kappa-4\big(\xi^2+(\eta+A\xi s)^2+\zeta^2\big)^2\big]\,ds\Big).
\]

We aim to obtain an estimate that improves
the integrability of $\widehat{\mathbb{G}}_2(t,\xi,\eta,\zeta)$.
Indeed, for any $0 < \varepsilon < \kappa$, we write
\begin{align}
\widehat{\mathbb{G}}_2(t,\xi,\eta,\zeta)
&=\exp\Bigg(\int_0^t\Big[
\big(\xi^2+(\eta+A\xi s)^2+\zeta^2\big)
-\kappa
-4\big(\xi^2+(\eta+A\xi s)^2+\zeta^2\big)^2
\Big]\,ds\Bigg) \nonumber\\[4pt]
&=e^{-\varepsilon t}\,
\exp\Bigg(\int_0^t\Big[
\xi^2+(\eta+A\xi s)^2+\zeta^2
-(\kappa-\varepsilon)
-4\big(\xi^2+(\eta+A\xi s)^2+\zeta^2\big)^2
\Big]\,ds\Bigg). \nonumber
\end{align}

Let
\[
v(s)=\xi^2+(\eta+A\xi s)^2+\zeta^2 .
\]
To deduce the estimate
\begin{equation}\label{eq:gestimates}
\widehat{\mathbb{G}}_2(t,\xi,\eta,\zeta)
\le e^{-\varepsilon t}
\exp\Bigg(-\int_0^t
C_0\,v(s)^2\,ds\Bigg),
\end{equation}
it suffices to require that
\[
v(s)-(\kappa-\varepsilon)-4v(s)^2
\le -C_0\,v(s)^2,
\qquad \forall\, s\ge0,
\]
which is equivalent to the quadratic condition
\[
(C_0-4)v(s)^2+v(s)-(\kappa-\varepsilon)\le0,
\qquad \forall\, v(s)\ge0.
\]

Assume $C_0<4$. Then the quadratic polynomial in $v(s)$ attains its maximum at
\[
v(s)=\frac{1}{2(4-C_0)}.
\]
Requiring this maximum to be non-positive yields
\[
0<C_0\le 4-\frac{1}{4(\kappa-\varepsilon)}.
\]
Hence such a positive constant $C_0$ exists provided that
\[
4-\frac{1}{4(\kappa-\varepsilon)}>0
\quad\Longleftrightarrow\quad
\kappa-\varepsilon>\frac{1}{16}.
\]
Therefore, it is sufficient to assume
\[
\kappa>\frac{1}{16}+\varepsilon.
\]

Throughout this paper, we fix $\varepsilon>0$ sufficiently small and assume that $\kappa>\tfrac{1}{16}+\varepsilon$.
\subsection{\texorpdfstring{$L^p$}{Lp} Estimates for the Green's Function ($p \ge 2$)}
In what follows, we focus on deriving $L^p$ estimates for the Green's function.
When $p \ge 2$, the desired estimates can be obtained easily by using the Fourier
transform of the Green's function together with Young's inequality.

We first state a key technical lemma that will play an essential role in the subsequent analysis, its proof can be found in \cite{mor}.
\begin{lemma}\label{lem:fourth}
For any $\alpha>0$, there exists a constant $C_\alpha>0$ such that
\begin{equation}\label{eq:lem}
\int_0^t|\eta+As\xi|^{\alpha}\,ds\geq C_\alpha\big(|\eta|^{\alpha}+(At)^{\alpha}|\xi|^{\alpha}\big)t.
\end{equation}
\end{lemma}

A direct calculation shows that
\begin{equation*}
\frac{1}{2}\big(\xi^4+(\eta+A\xi s)^4+\zeta^4\big)
\leq \big(\xi^2+(\eta+A\xi s)^2+\zeta^2\big)^2
\leq 4\big(\xi^4+(\eta+A\xi s)^4+\zeta^4\big).
\end{equation*}

Then by \eqref{eq:gestimates} and $\eqref{eq:lem}$, we have
\begin{equation}\label{eq:green-0}
\exp\Big(-C_0\displaystyle \int_0^t{{\big(\xi^2+(\eta+A\xi s)^2+\zeta^2\big)^2}}ds\Big)\leq \exp\Big(-\tilde{C}_0\left[(1+(At)^4)\xi^4t+\eta^4t+\zeta^4t\right]\Big).
\end{equation}
\begin{lemma}\label{thm:gdecay}
The following estimates for the Green's function are available for $p\ge2$,
\begin{align}
&\left\|{{\mathbb{G}}}(t, \cdot-x^\prime, \cdot, \cdot-z^\prime;y^\prime)\right\|_{L^p}
\leq C e^{-\varepsilon t} \, t^{-\frac{3}{4}(1-\frac{1}{p})}\left(1+(At)^4\right)^{-\frac{1}{4}(1-\frac{1}{p})},\label{eq:glp}
\\[6pt]&\left\|{{\mathbb{G}}}(t, x-\cdot, y, z-\cdot; \cdot)\right\|_{L^p}
\leq C e^{-\varepsilon t}\,t^{-\frac{3}{4}(1-\frac{1}{p})}\left(1+(At)^4\right)^{-\frac{1}{4}(1-\frac{1}{p})}.\label{eq:glpprime}
\end{align}
\end{lemma}
\begin{proof}
For $2 \le p \le \infty$,  and $\frac{1}{p}+\frac{1}{q}=1$, by Hausdorff-Young inequality,
\[
\left\|\mathbb{G}(t, \cdot-x^\prime, \cdot, \cdot-z^\prime; y^\prime)\right\|_{L^p}
\le
C \left\|
\widehat{\mathbb{G}}\big(t, \cdot, \cdot, \cdot; x',y',z'\big)
\right\|_{L^q_{\xi,\eta,\zeta}}.
\]

In fact, by \eqref{eq:gestimates} and \eqref{eq:green-0}, we have
\begin{align*}
 &\quad\Big\|
\widehat{\mathbb{G}}\big(t, \cdot,\cdot, \cdot; x', y', z'\big)
\Big\|_{L^q_{\xi,\eta,\zeta}}^q
=\int_{\mathbb{R}^3}\left|\widehat{\mathbb{G}}_1(t,\xi,\eta,\zeta; x^\prime, y^\prime, z^\prime)\cdot\widehat{\mathbb{G}}_2(t,\xi,\eta,\zeta)\right|^qd\xi\, d\eta\, d\zeta
 \\[4pt]&\leq e^{-\varepsilon q t}\,\int_{\mathbb{R}^3}
\exp\!\big(
-q \tilde{C}_0 \big[(1+(At)^4)\xi^4t + \eta^4t + \zeta^4t\big]
\big)
\, d\xi\, d\eta\, d\zeta
 \\[4pt]&\le
C\,e^{-\varepsilon qt}\,t^{-\frac{3}{4}}\big(1+(At)^4\big)^{-\frac{3}{4}},
\end{align*}
which implies,
\begin{align*}
\left\|\mathbb{G}(t, \cdot-x^\prime, \cdot, \cdot-z^\prime; y^\prime)\right\|_{L^p}\leq C e^{-\varepsilon t}\, t^{-\frac{3}{4}(1-\frac{1}{p})}\big(1+(At)^4\big)^{-\frac{3}{4}(1-\frac{1}{p})}.
\end{align*}
Thus, we obtain \eqref{eq:glp}. To prove \eqref{eq:glpprime}, we study the relationship between
$(x, y, z)$ and $(x^\prime, y^\prime, z^\prime)$ in Green's function $\mathbb{G}(t, x-x^\prime, y, z-z^\prime; y^\prime)$. Taking the inverse Fourier transform on Green's function, we have
\begin{align}\label{eq:greenexplicite}
  &\quad\mathbb{G}(t, x-x^\prime, y, z-z^\prime; y^\prime)\nonumber
 \\[4pt]&=\int_{\mathbb{R}^3}e^{i(x\xi+y\eta+z\zeta)}\,e^{-i x'\xi-i y^\prime(\eta+A\xi t)-i z'\zeta}\,e^{-\int_0^t \big[4\big(\xi^2+(\eta+A\xi s)^2+\zeta^2\big)^2-\big(\xi^2+(\eta+A\xi s)^2+\zeta^2\big)
+\kappa\big]\,ds}d\xi\,d\eta\,d\zeta\nonumber
\\[4pt]&\triangleq \mathbb{H}(t, x-x^\prime-Aty^\prime, y-y^\prime, z-z^\prime).
\end{align}
Observing the structure of
$\mathbb{H}(t, x-x^\prime-At y^\prime, y-y^\prime, z-z^\prime)$,
we see that the estimate \eqref{eq:glpprime} can be obtained immediately. This completes the proof of Lemma $\ref{thm:gdecay}$.
\end{proof}

We now turn to estimates for higher-order derivatives of the Green's function.
The following lemma addresses this issue.
\begin{lemma}\label{thm:decay-11}
For any $p\ge2$, the Green's function has the following estimates,
\begin{align*}
&\left\|\partial_x^{k_1}\partial_y^{k_2}\partial_z^{k_3}{{\mathbb{G}}}(t, \cdot-x^\prime, \cdot, \cdot-z^\prime; y^\prime)\right\|_{L^p}
\leq Ce^{-\varepsilon t}\,t^{-\frac{3}{4}(1-\frac{1}{p})-\frac{k}{4}}\left(1+(At)^4\right)^{-\frac{1}{4}(1-\frac{1}{p})-\frac{k_1}{4}},
\\[6pt]&\left\|\partial_{x}^{k_1}\partial_{y}^{k_2}\partial_{z}^{k_3}{{\mathbb{G}}}(t, x-\cdot, y, z-\cdot; \cdot)\right\|_{L^p}
\leq Ce^{-\varepsilon t}\,t^{-\frac{3}{4}(1-\frac{1}{p})-\frac{k}{4}}\left(1+(At)^4\right)^{-\frac{1}{4}(1-\frac{1}{p})-\frac{k_1}{4}},
\end{align*}
where $k_1, k_2, k_3$ are non-negative integers and $\ k=k_1+k_2+k_3$.
\end{lemma}
\begin{proof}
 From the definition of Green's function, by \eqref{eq:gestimates} and \eqref{eq:green-0}, it follows that
 \begin{equation*}
  \begin{split}
     &\quad\Big\|
\,|\xi|^{k_1}|\eta|^{k_2}|\zeta|^{k_3}
\,\widehat{\mathbb{G}}\big(t, \xi, \eta, \zeta; x', y', z'\big)
\Big\|_{L^1_{\xi,\eta,\zeta}}
     \\[4pt]&\leq e^{-\varepsilon t}\,\int_{\mathbb{{R}}^3}|\xi|^{k_1}|\eta|^{k_2}|\zeta|^{k_3}\exp\big(-\tilde{C}_0\left[(1+(At)^4)\xi^4t+\eta^4t+\zeta^4t\right]\big)d\xi d\eta d \zeta
     \\[4pt]&\leq Ce^{-\varepsilon t}\,t^{-\frac{3}{4}-\frac{k}{4}}(1+(At)^4)^{-\frac{k_1}{4}-\frac{1}{4}},
  \end{split}
  \end{equation*}
and
\begin{equation*}
  \begin{split}
 &\quad \Big\||\xi|^{k_1}|\eta|^{k_2}|\zeta|^{k_3}\widehat{\mathbb{G}}(t, \xi, \eta, \zeta; x^\prime, y^\prime, z^\prime)\Big\|_{L^2_{\xi, \eta, \zeta}}^2
\\[4pt]&\leq e^{-\varepsilon t}\,\int_{\mathbb{{R}}^3}|\xi|^{2k_1}|\eta|^{2k_2}|\zeta|^{2k_3}\exp\big(-2\tilde{C}_0\left[(1+(At)^4)\xi^4t+\eta^4t+\zeta^4t\right]\big) d\xi d\eta d\zeta
\\[4pt]&\leq Ce^{-\varepsilon t}\,t^{-\frac{3}{4}-\frac{k}{2}}(1+(At)^4)^{-\frac{k_1}{2}-\frac{1}{4}}.
  \end{split}
  \end{equation*}
Therefore, we can obtain the following estimates,
\begin{align*}
&\Big\||\xi|^{k_1}|\eta|^{k_2}|\zeta|^{k_3}\widehat{\mathbb{G}}(t, \xi, \eta, \zeta; x^\prime, y^\prime, z^\prime)\Big\|_{L^1_{\xi,\eta,\zeta}}\leq Ce^{-\varepsilon t}\,t^{-\frac{3}{4}-\frac{k}{4}}(1+(At)^4)^{-\frac{k_1}{4}-\frac{1}{4}}.
\\[6pt]
&\Big\||\xi|^{k_1}|\eta|^{k_2}|\zeta|^{k_3}\widehat{\mathbb{G}}(t, \xi, \eta, \zeta; x^\prime, y^\prime, z^\prime)\Big\|_{L^2_{\xi,\eta,\zeta}}\leq Ce^{-\varepsilon t}\,t^{-\frac{3}{8}-\frac{k}{4}}(1+(At)^4)^{-\frac{k_1}{4}-\frac{1}{8}}.
\end{align*}
Using the interpolation theorem and Young's inequality, we obtain the following key estimate for ${\mathbb{G}}(t, x-x^\prime, y, z-z^\prime; y^\prime)$,
\begin{equation*}
  \begin{array}{ll}
&\quad\Big\|\partial_x^{k_1}\partial_y^{k_2}\partial_z^{k_3} \mathbb{G}(t, \cdot-x^\prime, \cdot, \cdot-z^\prime; y^\prime)\Big\|_{L^p}
\\[8pt]&\leq \Big\|\partial_x^{k_1}\partial_y^{k_2} \partial_z^{k_3}\mathbb{G}(t, \cdot-x^\prime, \cdot, \cdot-z^\prime; y^\prime)\Big\|_{L^2}^{\frac{2}{p}}\ \cdot\ \Big\|\partial_x^{k_1}\partial_y^{k_2}\partial_z^{k_3} \mathbb{G}(t, \cdot-x^\prime, \cdot, \cdot-z^\prime; y^\prime)\Big\|_{L^\infty}^{1-\frac{2}{p}}
\\[8pt]&\leq \Big\||\xi|^{k_1}|\eta|^{k_2}|\zeta|^{k_3}\widehat{\mathbb{G}}(t, \xi, \eta, \zeta; x^\prime, y^\prime, z^\prime)\Big\|_{L^2_{\xi, \eta, \zeta}}^{\frac{2}{p}}\ \cdot\ \Big\||\xi|^{k_1}|\eta|^{k_2}|\zeta|^{k_3}\widehat{\mathbb{G}}(t, \xi, \eta, \zeta; x^\prime, y^\prime, z^\prime)\Big\|_{L^1_{\xi, \eta, \zeta}}^{1-\frac{2}{p}}
 \\[12pt]&\leq Ce^{-\varepsilon t}\,t^{-\frac{3}{4}(1-\frac{1}{p})-\frac{k}{4}}\left(1+(At)^4\right)^{-\frac{1}{4}(1-\frac{1}{p})-\frac{k_1}{4}}.
  \end{array}
\end{equation*}

Similarly, we can obtain
\begin{align*}
  \Big\|\partial_{x}^{k_1}\partial_{y}^{k_2}\partial_{z}^{k_3}{{\mathbb{G}}}(t, x-\cdot, y, z-\cdot; \cdot)\Big\|_{L^p}
\leq Ce^{-\varepsilon t}\,t^{-\frac{3}{4}(1-\frac{1}{p})-\frac{k}{4}}\left(1+(At)^4\right)^{-\frac{1}{4}(1-\frac{1}{p})-\frac{k_1}{4}}.
\end{align*}

Therefore, for $p\geq 2$, the decay estimates for the Green's function can be derived by combining the decay estimates of its Fourier transform with Young's inequality. This concludes the proof of Lemma~\ref{thm:decay-11}.
\end{proof}
\subsection{\texorpdfstring{$L^p$}{Lp} Estimates for the Green's Function $(1\le p<2)$}
Due to the singular behavior of the Green's function at $t=0$, we first establish pointwise bounds for $\mathbb{G}_2(t,x,y,z)$,
from which corresponding $L^1$ estimates are derived, allowing us to control estimates in the later analysis. The following lemma summarizes these bounds.
\begin{lemma}\label{lem:G2-pointwise}
For any integer $N \ge 1$, the following pointwise estimates for $\mathbb{G}_2(t,x,y,z)$ hold:

\begin{equation*}
\left|\mathbb{G}_2(t,x,y,z)\right|
\le C\, e^{-\varepsilon t}\, t^{-\frac34} (1+(At)^4)^{-\frac14}
\left(
1 + \frac{x^2}{t^{\frac12}(1+(At)^4)^{\frac12}} + \frac{y^2+z^2}{t^{\frac12}}
\right)^{-N},
\end{equation*}

\vspace{2mm}
\noindent and for non-negative integers $k_1, k_2, k_3$ with $k = k_1 + k_2 + k_3\leq 6$,
\vspace{1mm}
\begin{equation*}
\left|\partial_x^{k_1}\partial_y^{k_2}\partial_z^{k_3} \mathbb{G}_2(t,x,y,z)\right|
\le C\, e^{-\varepsilon t}\, t^{-\frac34 - \frac{k}{4}}
(1+(At)^4)^{-\frac14 - \frac{k_1}{4}}
\Bigg(
1 + \frac{x^2}{t^{\frac12}(1+(At)^4)^{\frac12}} + \frac{y^2+z^2}{t^{\frac12}}
\Bigg)^{-N}.
\end{equation*}
\end{lemma}

\begin{proof}
Recall that
\[
\widehat{\mathbb{G}}_2(t,\xi,\eta,\zeta)
=\exp\left(\int_0^t \Big[
\xi^2 + (\eta + A\xi s)^2 + \zeta^2 - \kappa
- 4\big(\xi^2 + (\eta + A\xi s)^2 + \zeta^2\big)^2
\Big]\,ds\right).
\]

A straightforward computation shows that the first-order derivatives satisfy
\begin{align*}
|\partial_\xi \widehat{\mathbb{G}}_2|
&\le C\,\mathbf Q\, t^{\frac14}(1+(At)^4)^{\frac14} |\widehat{\mathbb{G}}_2|,\\
|\partial_\eta \widehat{\mathbb{G}}_2| + |\partial_\zeta \widehat{\mathbb{G}}_2|
&\le C\,\mathbf Q\, t^{\frac14} |\widehat{\mathbb{G}}_2|,
\end{align*}
where we define
\[
\mathbf Q := 1 + t^{\frac34} (1+(At)^4)^{\frac34} (\xi^2 + \eta^2 + \zeta^2)^{3/2}.
\]

More generally, for any non-negative integer $m \le 3$, there exists a constant $C>0$ such that
\begin{align}
|\partial_\xi^{2m} \widehat{\mathbb{G}}_2|
&\le C\, \mathbf Q^{2m}\, t^{\frac{2m}{4}} (1+(At)^4)^{\frac{2m}{4}} |\widehat{\mathbb{G}}_2|, \label{eq:partialx} \\
|\partial_\eta^{2m} \widehat{\mathbb{G}}_2| + |\partial_\zeta^{2m} \widehat{\mathbb{G}}_2|
&\le C\, \mathbf Q^{2m}\, t^{\frac{2m}{4}} |\widehat{\mathbb{G}}_2|. \label{eq:partialyz}
\end{align}

By Fourier inversion and \eqref{eq:partialx}, we obtain
\begin{align}
\big|x^{2m}\mathbb{G}_2(t,x,y,z)\big|
&\le
\int_{\mathbb R^3}
\big|\partial_{\xi}^{2m}\widehat{\mathbb{G}}_2\big|
\,d\xi  d\eta  d\zeta\nonumber
\\[4pt]
&\le Ce^{-\varepsilon t}\,t^{-\frac34}(1+(At)^4)^{-\frac14}
t^{\frac{2m}{4}}(1+(At)^4)^{\frac{2m}{4}}.\label{eq:point-x}
\end{align}

Similarly, by \eqref{eq:partialyz}, we have
\begin{align}\label{eq:point-z}
\big|y^{2m}\mathbb{G}_2(t,x,y,z)\big|
+\big|z^{2m}\mathbb{G}_2(t,x,y,z)\big|
\le
Ce^{-\varepsilon t}\,t^{-\frac34}(1+(At)^4)^{-\frac14}
t^{\frac{2m}{4}}.
\end{align}

Combining~\eqref{eq:point-x} and \eqref{eq:point-z}, we obtain
\begin{align*}
& |{\mathbb{G}_2(t, x, y, z)}| \le Ce^{-\varepsilon t} \, t^{-\frac{3}{4}} \bigl(1+(At)^4\bigr)^{-\frac{1}{4}}
\left(\frac{x^2}{t^{\frac12} (1+(At)^4)^{\frac12}}\right)^{-m},\\[5mm]
& |{\mathbb{G}_2(t, x, y, z)}| \le Ce^{-\varepsilon t}\, t^{-\frac{3}{4}} \bigl(1+(At)^4\bigr)^{-\frac{1}{4}}
\left(\frac{y^2+z^2}{t^{\frac12}}\right)^{-m}.
\end{align*}

We observe that
\begin{align*}
1 + \frac{x^2}{t^{\frac12} (1+(At)^4)^{\frac12}} + \frac{y^2+z^2}{t^{\frac12}}
\le 4 \, \begin{cases}
1, & (x,y,z)\in \Omega_1,\\[8pt]
\dfrac{y^2+z^2}{t^{\frac12}}, & (x,y,z)\in \Omega_2,\\[8pt]
\dfrac{x^2}{t^{\frac12} (1+(At)^4)^{\frac12}}, & (x,y,z)\in \Omega_3,\\[10pt]
\dfrac{x^2}{t^{\frac12} (1+(At)^4)^{\frac12}} + \dfrac{y^2+z^2}{t^{\frac12}}, & (x,y,z)\in \Omega_4,
\end{cases}
\end{align*}
where
\begin{align*}
\Omega_1 &= \Big\{ x^2 \le t^{\frac12} (1+(At)^4)^{\frac12},\; \ y^2+z^2 \le t^{\frac12} \Big\},\\[4pt]
\Omega_2 &= \Big\{ x^2 \le t^{\frac12} (1+(At)^4)^{\frac12},\; \ y^2+z^2 > t^{\frac12} \Big\},\\[4pt]
\Omega_3 &= \Big\{ x^2 > t^{\frac12} (1+(At)^4)^{\frac12},\; \ y^2+z^2 \le t^{\frac12} \Big\},\\[4pt]
\Omega_4 &= \Big\{ x^2 > t^{\frac12} (1+(At)^4)^{\frac12},\; \ y^2+z^2 > t^{\frac12} \Big\}.
\end{align*}

Consequently, for any positive integer $N\ge1$, there exists a constant $C$, such that
\[
\big|\mathbb{G}_2(t,x,y,z)\big|
\le
Ce^{-\varepsilon t}\,t^{-\frac34}(1+(At)^4)^{-\frac14}
\left(
1+\frac{x^2}{t^{\frac12}(1+(At)^4)^{\frac12}}
+\frac{y^2+z^2}{t^{\frac12}}
\right)^{-N}.
\]

Moreover, for non-negative integers $k_1, k_2, k_3$ with
$k = k_1 + k_2 + k_3 \le 6$, there exists a constant $C>0$ such that
\[
\Big|\partial_x^{k_1}\partial_y^{k_2}\partial_z^{k_3}\mathbb{G}_2(t,x,y,z)\Big|
\le
Ce^{-\varepsilon t}\,t^{-\frac34-\frac{k}{4}}
(1+(At)^4)^{-\frac14-\frac{k_1}{4}}
\left(
1+\frac{x^2}{t^{\frac12}(1+(At)^4)^{\frac12}}
+\frac{y^2+z^2}{t^{\frac12}}
\right)^{-N}.
\]
This completes the proof of Lemma~\ref{lem:G2-pointwise}.
\end{proof}

From the pointwise estimates, we observe the presence of an additional factor
$(1+(At)^4)^{-1/4}$ in $\mathbb{G}_2(t,x,y,z)$.
This factor provides clear evidence of enhanced dissipation induced by the
Couette flow.
As a consequence of the pointwise estimates for $\mathbb{G}_2(t,x,y,z)$,
we now derive the $L^p$-norm estimates for $1\le p<2$,
\begin{equation*}
\begin{split}
\left\|\partial_x^{k_1}\partial_y^{k_2}\partial_z^{k_3}
\mathbb{G}_2(t,\cdot,\cdot,\cdot)\right\|_{L^p}
&\le Ce^{-\varepsilon t}\, t^{-\frac{3}{4}-\frac{k}{4}}
(1+(At)^4)^{-\frac{1}{4}-\frac{k_1}{4}}
\\[4pt]
&\quad\cdot
\left(
\int_{\mathbb{R}^3}
\Big(
1+\frac{x^2}{t^{\frac12}(1+(At)^4)^{\frac12}}
+\frac{y^2+z^2}{t^{\frac12}}
\Big)^{-Np}
\,dx\,dy\,dz
\right)^{\frac1p}.
\end{split}
\end{equation*}

Consequently, we obtain the following lemma
\begin{lemma}\label{thm:greenlp}
Let $1 \le p < 2$, and let $k_1, k_2, k_3$ be non-negative integers with $k = k_1 + k_2 + k_3$.
Then there exists a constant $C>0$ such that
\begin{align*}
\left\|\partial_x^{k_1}\partial_y^{k_2}\partial_z^{k_3}
\mathbb{G}_2(t,\cdot,\cdot,\cdot)\right\|_{L^p}
\le C\, e^{-\varepsilon t}\,
t^{-\frac{3}{4}\left(1-\frac{1}{p}\right)-\frac{k}{4}}
\left(1 + (At)^4 \right)^{-\frac{1}{4}\left(1-\frac{1}{p}\right) - \frac{k_1}{4}}.
\end{align*}
\end{lemma}

Combining Lemma~\ref{thm:decay-11} and Lemma~\ref{thm:greenlp}, we derive $L^p$ estimates
for derivatives of the Green's function for all $p \ge 1$.
\begin{prop}\label{prop:green-lp}
Let $k_1, k_2, k_3$ be non-negative integers with $k = k_1 + k_2 + k_3\leq 6$.
Then the Green's function
$\mathbb{G}(t, x-x^\prime, y, z-z^\prime; y^\prime)$
satisfies the following $L^p$ estimates for all $p \ge 1$:
\begin{align}
\left\|\partial_x^{k_1}\partial_y^{k_2}\partial_z^{k_3}
\mathbb{G}(t, \cdot-x^\prime, \cdot, \cdot-z^\prime; y^\prime)\right\|_{L^p}
&\le C\, e^{-\varepsilon t}\,
t^{-\frac{3}{4}\left(1-\frac{1}{p}\right)-\frac{k}{4}}
\left(1 + (At)^4 \right)^{-\frac{1}{4}\left(1-\frac{1}{p}\right) - \frac{k_1}{4}},
\label{eq:greenx-x} \\[2mm]
\left\|\partial_x^{k_1}\partial_y^{k_2}\partial_z^{k_3}
\mathbb{G}(t, x-\cdot, y, z-\cdot; \cdot)\right\|_{L^p}
&\le C\, e^{-\varepsilon t}\,
t^{-\frac{3}{4}\left(1-\frac{1}{p}\right)-\frac{k}{4}}
\left(1 + (At)^4 \right)^{-\frac{1}{4}\left(1-\frac{1}{p}\right) - \frac{k_1}{4}},
\label{eq:greenx-xprime} \\[2mm]
\left\|\partial_{x^\prime}^{k_1}\partial_{y^\prime}^{k_2}\partial_{z^\prime}^{k_3}
\mathbb{G}(t, \cdot-x^\prime, \cdot, \cdot-z^\prime; y^\prime)\right\|_{L^p}
&\le C\, e^{-\varepsilon t}\,
t^{-\frac{3}{4}\left(1-\frac{1}{p}\right)-\frac{k}{4}}
\left(1 + (At)^4 \right)^{-\frac{1}{4}\left(1-\frac{1}{p}\right) - \frac{k_1}{4}},
\label{eq:greenxprime-x} \\[2mm]
\left\|\partial_{x^\prime}^{k_1}\partial_{y^\prime}^{k_2}\partial_{z^\prime}^{k_3}
\mathbb{G}(t, x-\cdot, y, z-\cdot; \cdot)\right\|_{L^p}
&\le C\, e^{-\varepsilon t}\,
t^{-\frac{3}{4}\left(1-\frac{1}{p}\right)-\frac{k}{4}}
\left(1 + (At)^4 \right)^{-\frac{1}{4}\left(1-\frac{1}{p}\right) - \frac{k_1}{4}}.
\label{eq:greenxprime-xprime}
\end{align}
Here, the $L^p$ norms are taken with respect to the spatial variables represented by the dot notation,
and $C>0$ denotes a constant independent of $t$.
\end{prop}

\begin{proof}
Estimates \eqref{eq:greenx-x} and \eqref{eq:greenx-xprime} follow directly from the $L^p$ estimates established in Lemma~\ref{thm:decay-11} and Lemma~\ref{thm:greenlp}.

It remains to prove \eqref{eq:greenxprime-x} and \eqref{eq:greenxprime-xprime}. From the explicit formula of the Green's function, we observe that $\mathbb{G}(t, x-x^\prime, y, z-z^\prime; y^\prime)$ is not symmetric with respect to $(x, y, z)$ and $(x', y', z')$. A direct computation gives
{\small{\begin{align}
\left\{\begin{array}{ll}
& \partial_x \mathbb{G}(t, x-x^\prime, y, z-z^\prime; y^\prime)=-\partial_{x^\prime} \mathbb{G}(t, x-x^\prime, y, z-z^\prime; y^\prime),
\\[4pt]& \partial_y \mathbb{G}(t, x-x^\prime, y, z-z^\prime; y^\prime)=-\partial_{y^\prime} \mathbb{G}(t, x-x^\prime, y, z-z^\prime; y^\prime)+At\partial_{x^\prime} \mathbb{G}(t, x-x^\prime, y, z-z^\prime; y^\prime),
\\[4pt]& \partial_z \mathbb{G}(t, x-x^\prime, y, z-z^\prime; y^\prime)=-\partial_{z^\prime} \mathbb{G}(t, x-x^\prime, y, z-z^\prime; y^\prime).
\end{array}\right.\label{eq:transfer}
\end{align}}}
Using these relations, the estimates for derivatives with respect to $(x', y', z')$ follow immediately from the estimates for $(x, y, z)$. The proof of Proposition~\ref{prop:green-lp} is complete.
\end{proof}

\begin{remark}
The above equations in $\eqref{eq:transfer}$ indicate that when attempting to convert the derivative of the Green's function with respect to
$y$ into a derivative with respect to $y^\prime$, an additional factor $At$ arises.
\end{remark}

By Duhamel's principle, the solution to \eqref{eq:mainequation} can be represented as
\begin{equation}\label{eq:duhamel}
\begin{split}
   {\phi}(t, x, y, z)&=\displaystyle\int_{\mathbb{{R}}^3}\mathbb{G}(t, x-x^\prime, y, z-z^\prime; y^\prime)\,\phi_0(x^\prime, y^\prime, z^\prime)dx^\prime dy^\prime dz^\prime\nonumber
\\[4pt]&\quad+\displaystyle\int^t_0\displaystyle\int_{\mathbb{{R}}^3}\mathbb{G}(t-s, x-x^\prime, y, z-z^\prime; y^\prime)\, \left(-\tfrac{1}{2}|\nabla_{\scriptscriptstyle{x^\prime, y^\prime, z^\prime}}\,\phi|^2(s, x^\prime, y^\prime, z^\prime)\right)dx^\prime dy^\prime dz^\prime ds.
\end{split}
\end{equation}
The above analysis shows that the Green's function $\mathbb{G}(t, x-x^\prime, y, z-z^\prime; y^\prime)$ acts as an integral kernel.
Due to the presence of the shear term $Ay\partial_x$, the kernel is not symmetric with respect to
$(x, y, z)$ and $(x', y', z')$, and the associated linear operator is therefore not of convolution type.
As a consequence, its action on a function cannot be represented as a simple convolution in physical space. For the convenience of later writing, we introduce the following notations,
\begin{equation*}
\mathbb{G}(t, x, y, z)\circledast f(x, y, z)=\int_{{\mathbb{R}^3}}\mathbb{G}(t, x-x^\prime, y, z-z^\prime; y^\prime)f(x^\prime, y^\prime, z^\prime)dx^\prime dy^\prime dz^\prime,
\end{equation*}
and
\begin{align}\label{eq:greennorm}
\interleave\mathbb{G}(t)\interleave_{L^p}=\max\limits_{
x, y, z,
x^\prime, y^\prime, z^\prime,}
 \Big\{\left\|\mathbb{G}(t, x-\cdot, y, z-\cdot;y^\prime)\right\|_{L^p},\ \left\|\mathbb{G}(t, \cdot-x^\prime, \cdot, \cdot-z^\prime; y^\prime)\right\|_{L^p}\Big\}.
\end{align}

For the convenience of later discussion, we summarize the results of this section in the following theorem.

\begin{thm}\label{thm:greenfunction}
The Green's function $\mathbb{G}(t, x-x^\prime, y, z-z^\prime; y^\prime)$ satisfies the following estimates for all $p\ge 1$:
\begin{align}
  &\interleave \partial_x^{k_1}\partial_y^{k_2}\partial_z^{k_3} \mathbb{G}(t) \interleave_{L^p}
    \le C e^{-\varepsilon t}\, t^{-\frac{3}{4}\left(1-\frac{1}{p}\right)-\frac{k}{4}}
    \left( 1 + (At)^4 \right)^{-\frac{1}{4}\left(1-\frac{1}{p}\right)-\frac{k_1}{4}}, \label{eq:green-deriv-x} \\[2mm]
  &\interleave \partial_{x^\prime}^{k_1}\partial_{y^\prime}^{k_2}\partial_{z^\prime}^{k_3} \mathbb{G}(t) \interleave_{L^p}
    \le C e^{-\varepsilon t}\, t^{-\frac{3}{4}\left(1-\frac{1}{p}\right)-\frac{k}{4}}
    \left( 1 + (At)^4 \right)^{-\frac{1}{4}\left(1-\frac{1}{p}\right)-\frac{k_1}{4}}. \label{eq:green-deriv-xprime}
\end{align}
Here, $\interleave \cdot \interleave_{L^p}$ is defined as in \eqref{eq:greennorm}.
\end{thm}
\section{Global Existence of Solutions}
\subsection{Local Existence of the Solution}
The local existence of solutions to \eqref{eq:mainequation} can be obtained by standard methods. For the convenience of the reader, we state the theorem without providing the proof.
\begin{thm}\label{thm:local}
Let the initial data satisfy
\[
\phi_0 \in W^{4,\infty}(\mathbb{R}^3) \cap L^1(\mathbb{R}^3).
\]
Then there exists a time $T_0 = T(\phi_0) > 0$ such that \eqref{eq:mainequation} admits a solution
\[
\phi(t,x,y,z) \in C([0,T_0]; W^{4,\infty}(\mathbb{R}^3) \cap L^1(\mathbb{R}^3)).
\]

Furthermore, if the solution remains bounded in $L^\infty$ up to time $T_0$, that is,
\[
\sup_{0 \le t \le T_0} \|\phi(t)\|_{L^\infty} < \infty,
\]
then there exists $\varepsilon>0$ such that a solution exists on $[0, T_0+\varepsilon]$.
\end{thm}

\subsection{Global Existence of the Solution}
In what follows, we employ a bootstrap argument to establish the global existence
of solutions to equation~\eqref{eq:mainequation}. We fix a constant $\delta>0$
and impose the following bootstrap hypothesis:
\begin{equation}\label{eq:re-1}
\|\nabla \phi(t)\|_{L^\infty}\leq 2\delta\, e^{-\varepsilon t}\,t^{-\frac{1}{3}}(1+t)^{-\frac{5}{3}},\qquad
\|\nabla \phi(t)\|_{L^2}\leq 2\delta\, e^{-\varepsilon t}\,t^{-\frac{1}{3}}(1+t)^{-\frac{19}{24}}.
\end{equation}
Here, $\delta = 10\max\{\|\phi_0\|_{L^1},\, \|\phi_0\|_{L^\infty}\}$.

Our goal is to verify~\eqref{eq:re-1} via a bootstrap argument based on the
Green's function estimates. As a first step, we establish the following lemma.

\begin{lemma}\label{lem:boots}
Suppose that the solution $\phi(t,x,y,z)$ of equation~\eqref{eq:mainequation}
satisfies the bootstrap hypothesis~\eqref{eq:re-1}. If $A$ is chosen sufficiently
large, then the improved estimates
\begin{equation*}
\|\nabla \phi(t)\|_{L^\infty}\leq \delta\, e^{-\varepsilon t}\,t^{-\frac{1}{3}}(1+t)^{-\frac{5}{3}},\qquad
\|\nabla \phi(t)\|_{L^2}\leq \delta\, e^{-\varepsilon t}\,t^{-\frac{1}{3}}(1+t)^{-\frac{19}{24}}
\end{equation*}
hold for all $t>0$.
\end{lemma}

\begin{proof}
By Duhamel's principle, the solution to equation~\eqref{eq:mainequation} admits the representation
\begin{equation*}
\phi(t,x,y,z)
= \mathbb{G}(t)\circledast \phi_0
+ \int_0^t \mathbb{G}(t-s)\circledast
\Big(-\tfrac{1}{2} \big|\nabla_{x', y', z'}\,\phi(s)\big|^2\Big)\,ds .
\end{equation*}
Taking the gradient yields
\begin{equation*}
\begin{aligned}
\|\nabla\phi(t)\|_{L^\infty}
\le\;&
\|\nabla_{x,y,z}(\mathbb{G}(t)\circledast \phi_0)\|_{L^\infty}  \\
&+
\Big\|
\int_0^t
\nabla_{x,y,z}
\Big(
\mathbb{G}(t-s)\circledast
\big(-\tfrac{1}{2} |\nabla_{x',y',z'}\,\phi(s)|^2\big)
\Big)\,ds
\Big\|_{L^\infty}.
\end{aligned}
\end{equation*}

For the linear term, by Theorem~\ref{thm:greenfunction} and recalling that
$\delta = 10 \max\{\|\phi_0\|_{L^1},\ \|\phi_0\|_{L^\infty}\}$,
we obtain
\begin{equation*}
\begin{aligned}
\|\nabla_{x,y,z}(\mathbb{G}(t)\circledast \phi_0)\|_{L^\infty}
&\le
\interleave\nabla_{x,y,z}\mathbb{G}(t)\interleave_{L^\infty}\,\|\phi_0\|_{L^1} \\
&\le
C e^{-\varepsilon t}\, t^{-1}(1+(At)^4)^{-\frac{1}{4}}\|\phi_0\|_{L^1} \\
&\le
C\delta\, e^{-\varepsilon t}\, A^{-1} t^{-2}.
\end{aligned}
\end{equation*}
On the other hand, by Young's inequality and Theorem~\ref{thm:greenfunction},
\begin{equation*}
\begin{aligned}
\|\nabla_{x,y,z}(\mathbb{G}(t)\circledast \phi_0)\|_{L^\infty}
&\le
\interleave\nabla_{x,y,z}\mathbb{G}(t)\interleave_{L^{\frac{21}{20}}}\,
\|\phi_0\|_{L^{21}} \\
&\le
C\delta\, e^{-\varepsilon t}\, A^{-\frac{1}{21}} t^{-\frac{1}{3}}.
\end{aligned}
\end{equation*}
Combining the two estimates above and choosing $A$ sufficiently large, we conclude that
\begin{equation}\label{eq:infinity-0}
\|\nabla_{x,y,z}(\mathbb{G}(t)\circledast \phi_0)\|_{L^\infty}
\le
\frac{\delta}{5}\, e^{-\varepsilon t}\,
t^{-\frac{1}{3}}(1+t)^{-\frac{5}{3}}.
\end{equation}

Next, we estimate the nonlinear term. We split the time integral at $t/2$ and write
\begin{equation*}
\begin{aligned}
\Big\|
\int_0^t
\nabla_{x,y,z}
\Big(
\mathbb{G}(t-s)\circledast
\big(-\tfrac{1}{2} |\nabla_{x',y',z'}\ \phi|^2(s)\big)
\Big)\,ds
\Big\|_{L^\infty}
\le K_1 + K_2 ,
\end{aligned}
\end{equation*}
where
\[
\begin{aligned}
K_1 &:= \Big\|
\int_0^{t/2}
\nabla_{x,y,z}\mathbb{G}(t-s)\circledast
\big(-\tfrac{1}{2} |\nabla_{x',y',z'}\,\phi|^2(s)\big)\,ds
\Big\|_{L^\infty}, \\
K_2 &:= \Big\|
\int_{t/2}^{t}
\nabla_{x,y,z}\mathbb{G}(t-s)\circledast
\big(-\tfrac{1}{2} |\nabla_{x',y',z'}\,\phi|^2(s)\big)\,ds
\Big\|_{L^\infty}.
\end{aligned}
\]

To proceed with the nonlinear estimates, using the bootstrap hypothesis~\eqref{eq:re-1} and Theorem~\ref{thm:greenfunction}, we first estimate $K_1$ via the $L^\infty$--$L^1$ convolution inequality,
\begin{equation*}
\begin{aligned}
K_1
&\le
\int_0^{t/2}
\interleave\nabla_{x,y,z}\mathbb{G}(t-s)\interleave_{L^\infty}
\|\nabla_{x',y',z'}\phi(s)\|_{L^2}^2\,ds \\
&\le
C\delta^2 e^{-\varepsilon t}
\int_0^{t/2}
(t-s)^{-1}(1+(A(t-s))^4)^{-\frac{1}{4}}
s^{-\frac{2}{3}}(1+s)^{-\frac{19}{12}}\,ds \\
&\le
C\delta^2 e^{-\varepsilon t}\,A^{-1}
\int_0^{t/2}
(t-s)^{-2}s^{-\frac{2}{3}}(1+s)^{-\frac{19}{12}}\,ds \\
&\le
\frac{\delta}{5} e^{-\varepsilon t}\, t^{-2},
\end{aligned}
\end{equation*}
where in the last step we used the assumption that $A$ is sufficiently large
so that $C\delta A^{-1}<\frac{1}{5}$.

On the other hand, we also have
\begin{equation*}
\begin{aligned}
K_1
&\le
\int_0^{t/2}
\interleave\nabla_{x,y,z}\mathbb{G}(t-s)\interleave_{L^{\frac{21}{16}}}
\|\nabla_{x',y',z'}\,\phi(s)\|_{L^{\frac{42}{5}}}^2\,ds \\
&\le
C\delta^2 e^{-\varepsilon t}
\int_0^{t/2}
(t-s)^{-\frac{2}{3}}(1+(A(t-s))^4)^{-\frac{5}{84}}
s^{-\frac{2}{3}}\,ds \\
&\le
C\delta^2 e^{-\varepsilon t}\,A^{-\frac{5}{21}}
\int_0^{t/2}
(t-s)^{-\frac{2}{3}}s^{-\frac{2}{3}}\,ds \\
&\le
\frac{\delta}{5} e^{-\varepsilon t}\, t^{-\frac{1}{3}},
\end{aligned}
\end{equation*}
where in the last step we used the assumption that $A$ is sufficiently large
so that $C\delta A^{-\frac{5}{21}}<\frac{1}{5}$.

Collecting the above inequalities, we obtain
\begin{equation}\label{eq:infinity-1}
K_1
\le
\frac{\delta}{5} e^{-\varepsilon t}\,
t^{-\frac{1}{3}}(1+t)^{-\frac{5}{3}}.
\end{equation}

Similarly, for $K_2$,
\begin{align}
 &K_2\leq \int^t_{\frac{t}{2}}\interleave\nabla_{x,y,z}\mathbb{G}(t-s)\interleave_{L^{\frac{21}{16}}}\|\nabla_{x',y',z'} \,\phi(s)\|^2_{L^{\frac{42}{5}}}ds\nonumber
\\[2mm]&\leq C\delta^2e^{-\varepsilon t} A^{-\frac{5}{21}}t^{-\frac{2}{3}}(1+t)^{-\frac{35}{12}}\int_{\frac{t}{2}}^t (t-s)^{-\frac{2}{3}}ds\nonumber
\\[1mm]&\leq \frac{\delta}{5} e^{-\varepsilon t} t^{-\frac{1}{3}}(1+t)^{-\frac{5}{3}}.\label{eq:infinity-2}
\end{align}
where the largeness of $A$ ensures that
$C\delta A^{-\frac{5}{21}}<\frac{1}{5}$.

Combining~\eqref{eq:infinity-0}, \eqref{eq:infinity-1} and~\eqref{eq:infinity-2},  choosing $A$ sufficiently large, we conclude that
\[
\|\nabla\phi(t)\|_{L^\infty}
\le
\frac{\delta}{5} e^{-\varepsilon t}\,
t^{-\frac{1}{3}}(1+t)^{-\frac{5}{3}}.
\]

In what follows, we estimate the $L^2$-norm of $\nabla \phi(t)$. Similarly we have
\begin{equation}\label{eq:l2-nabla}
\begin{split}
\|\nabla \phi(t)\|_{L^2}
\le{}& \|\nabla_{x,y,z} (\mathbb{G}(t)\circledast \phi_0)\|_{L^2} \\
&+ \Big\|\int_0^t \nabla_{x,y,z} \Big(\mathbb{G}(t-s)\circledast \big(-\tfrac{1}{2}|\nabla_{x', y', z'}\,\phi|^2(s)\big)\Big)\,ds\Big\|_{L^2}.
\end{split}
\end{equation}

By Theorem~\ref{thm:greenfunction}, the linear term can be estimated as
\begin{align}
\|\nabla_{x,y,z} \big(\mathbb{G}(t)\circledast \phi_0\big)\|_{L^2}
&\le \interleave\nabla_{x,y,z} \mathbb{G}(t)\interleave_{L^2} \|\phi_0\|_{L^1} \nonumber\\
&\le C \delta e^{-\varepsilon t} A^{-\tfrac{1}{2}} t^{-\frac{9}{8}}, \label{eq:nabla-linear-L2}\\[1mm]
\|\nabla_{x,y,z} \big(\mathbb{G}(t)\circledast \phi_0\big)\|_{L^2}
&\le \interleave\nabla_{x,y,z} \mathbb{G}(t)\interleave_{L^{21/20}} \|\phi_0\|_{L^{42/23}} \nonumber\\
&\le C\delta e^{-\varepsilon t} A^{-\frac{1}{21}} t^{-\frac{1}{3}}. \label{eq:nabla-linear-Lp}
\end{align}
It follows from \eqref{eq:nabla-linear-L2} and \eqref{eq:nabla-linear-Lp} that
\begin{equation}\label{eq:nabla-linear-final}
\|\nabla_{x,y,z}\big(\mathbb{G}(t)\circledast \phi_0\big)\|_{L^2}
\le \frac{\delta}{5} e^{-\varepsilon t} t^{-\frac{1}{3}} (1+t)^{-\frac{19}{24}}.
\end{equation}
Next, we estimate the nonlinear term by splitting the time integral at $t/2$:
\begin{equation*}
\Big\|\int_0^t \nabla_{x,y,z} \Big(\mathbb{G}(t-s)\circledast (-\tfrac{1}{2}|\nabla_{{x', y', z'}} \phi|^2(s))\Big)\,ds\Big\|_{L^2}
=: M_1 + M_2,
\end{equation*}
where
\[
\begin{aligned}
M_1 &:= \Big\|
\int_0^{t/2}
\nabla_{x,y,z}\mathbb{G}(t-s)\circledast
\big(-\tfrac{1}{2} |\nabla_{x',y',z'}\,\phi|^2(s)\big)\,ds
\Big\|_{L^2}, \\
M_2 &:= \Big\|
\int_{t/2}^{t}
\nabla_{x,y,z}\mathbb{G}(t-s)\circledast
\big(-\tfrac{1}{2} |\nabla_{x',y',z'}\,\phi|^2(s)\big)\,ds
\Big\|_{L^2}.
\end{aligned}
\]

For $M_1$,  we have
\begin{align*}
M_1
&\le \int_0^{t/2} \interleave\nabla_{x,y,z} \mathbb{G}(t-s)\interleave_{L^2} \|\nabla_{{x', y', z'}}  \phi(s)\|_{L^2}^2 \, ds \\
&\le C \delta^2 e^{-\varepsilon t} A^{-\frac{1}{2}} \int_0^{t/2} (t-s)^{-\frac{9}{8}} s^{-\frac{2}{3}} (1+s)^{-\frac{19}{12}} \, ds \\
&\le \frac{\delta}{5} e^{-\varepsilon t} t^{-\frac{9}{8}},
\end{align*}
where in the last step we used the assumption that $A$ is sufficiently large
so that $C\delta A^{-\frac{1}{2}}<\frac{1}{5}$.

On the other hand, we also have
\begin{align*}
M_1
&\le \int_0^{t/2} \interleave\nabla_{x,y,z} \mathbb{G}(t-s)\interleave_{L^{\frac{21}{16}}} \|\nabla_{{x', y', z'}}  \phi(s)\|_{L^{\frac{84}{31}}}^2 \, ds \\
&\le C \delta^2 e^{-\varepsilon t} A^{-\frac{5}{21}} \int_0^{t/2} (t-s)^{-\frac{2}{3}} s^{-\frac{2}{3}} \, ds \\
&\le \frac{\delta}{5} e^{-\varepsilon t} t^{-\frac{1}{3}},
\end{align*}
where in the last step we used the assumption $C\delta A^{-\frac{5}{21}}<\frac{1}{5}$.

Hence,
\begin{equation}\label{eq:M1-final}
M_1 \le \frac{\delta}{5} e^{-\varepsilon t} t^{-\frac{1}{3}} (1+t)^{-\frac{19}{24}}.
\end{equation}

Similarly, for $M_2$ we have
\begin{align}\label{eq:M2-final}
M_2
&\le C \delta^2 e^{-\varepsilon t} A^{-\frac{5}{21}} \int_{t/2}^t (t-s)^{-\frac{2}{3}} s^{-\frac{2}{3}} (1+s)^{-\frac{19}{24}} \, ds \nonumber\\
&\le \frac{\delta}{5} e^{-\varepsilon t} t^{-\frac{1}{3}} (1+t)^{-\frac{19}{24}},
\end{align}
where in the last step we also used the assumption $C\delta A^{-\frac{5}{21}}<\frac{1}{5}$.

Finally, combining \eqref{eq:nabla-linear-final}, \eqref{eq:M1-final}, and \eqref{eq:M2-final}, and taking $A$ sufficiently large, we obtain
\begin{equation}\label{eq:nabla-final}
\|\nabla \phi(t)\|_{L^2}
\le \frac{\delta}{5} e^{-\varepsilon t} t^{-\frac{1}{3}} (1+t)^{-\frac{19}{24}}.
\end{equation}
This improves the bootstrap bounds and therefore closes the bootstrap argument.
\end{proof}

In the following, we show that $\eqref{eq:re-1}$ is a regularity criterion for the solution of the equation $\eqref{eq:mainequation}$. For
convenience, in the following proof, we write the $\delta$ occurring in the inequality $\eqref{eq:re-1}$  as the constant $C$. We only prove the inequality satisfied by $\|\nabla \phi(t)\|_{L^2}$, as the argument for $\|\nabla \phi(t)\|_{L^\infty}$ is similar. First, we show that under the bootstrap hypothesis $\eqref{eq:re-1}$, the following lemma holds,
\begin{lemma}\label{lem:regularity}
Suppose that $\phi(t,x,y,z)$ is a solution of equation \eqref{eq:mainequation} satisfying the bootstrap hypothesis \eqref{eq:re-1}. Then we have
\begin{equation*}
 \|\nabla \phi(t)\|_{L^2}\leq C e^{-\varepsilon t}\,(1+t)^{-\frac{9}{8}}, \qquad \|\nabla \phi(t)\|_{L^\infty}\leq C e^{-\varepsilon t}\,(1+t)^{-2}.
\end{equation*}
\end{lemma}

\begin{proof}
First, we show that under the bootstrap hypothesis \eqref{eq:re-1}, one has
\begin{equation}\label{eq:lem-deltau}
 \|\nabla \phi(t)\|_{L^2}\leq C e^{-\varepsilon t}\,t^{-\frac{1}{6}}(1+t)^{-\frac{23}{24}},\qquad \|\nabla \phi(t)\|_{L^\infty}\leq Ce^{-\varepsilon t}\, t^{-\frac{1}{6}}(1+t)^{-\frac{11}{6}}.
\end{equation}
By Duhamel's principle, $L^2$ norm of $\nabla \phi(t)$ can be expressed as follows:
\begin{equation*}
  \left\|\nabla \phi(t)\right\|_{L^2}
  \leq \left\|\nabla_{\scriptscriptstyle{x, y, z}}(\mathbb{G}(t)\circledast \phi_0)\right\|_{L^2}+\Big\|\int^t_0\nabla_{x, y, z} \mathbb{G}(t-s)\circledast(-\tfrac{1}{2}|\nabla_{\scriptscriptstyle{x^\prime, y^\prime, z^\prime}} \phi|^2(s))ds\Big\|_{L^2}.
\end{equation*}

For the linear part, by $\eqref{eq:transfer}$,  we have
  \begin{equation*}
\begin{split}
 &\quad\|\nabla_{\scriptscriptstyle{x, y, z}}(\mathbb{G}(t)\circledast \phi_0)\|_{L^2}
 \leq \interleave\mathbb{G}(t)\interleave_{L^{\frac{21}{19}}}(1+At)\|\nabla_{\scriptscriptstyle{x^\prime, y^\prime, z^\prime}} \phi_0\|_{L^{\frac{42}{25}}}
\leq Ce^{-\varepsilon t}\, t^{-\frac{1}{6}}(1+t).
\end{split}
\end{equation*}
On the other hand, we also have
\begin{align*}
&\quad\|\nabla_{\scriptscriptstyle{x, y, z}}(\mathbb{G}(t)\circledast \phi_0)\|_{L^2}\leq \interleave\nabla_{\scriptscriptstyle{x, y, z}}\mathbb{G}(t)\interleave_{L^2}\|\phi_0\|_{L^1}
\leq Ce^{-\varepsilon t}\,A^{-\frac{1}{2}}t^{-\frac{9}{8}}\leq Ce^{-\varepsilon t}\,t^{-\frac{9}{8}}.
\end{align*}
Therefore,
\begin{equation}\label{eq:l222-1}
\|\nabla_{\scriptscriptstyle{x, y, z}}(\mathbb{G}(t)\circledast \phi_0)\|_{L^2}\leq Ce^{-\varepsilon t}\,t^{-\frac{1}{6}}(1+t)^{-\frac{23}{24}}.
\end{equation}

For the nonlinear part, we have
\begin{align*}
  &\quad\Big\|\int^t_0\nabla_{x, y, z} \mathbb{G}(t-s)\circledast(-\tfrac{1}{2}|\nabla_{\scriptscriptstyle{x^\prime, y^\prime, z^\prime}} \phi|^2(s))ds\Big\|_{L^2}
  \\[4pt]&\leq \Big\|\int^{\frac{t}{2}}_0\nabla_{x, y, z} \mathbb{G}(t-s)\circledast(-\tfrac{1}{2}|\nabla_{\scriptscriptstyle{x^\prime, y^\prime, z^\prime}} \phi|^2(s))ds\Big\|_{L^2}
  \\[4pt]&\quad+\Big\|\int^t_{\frac{t}{2}}\nabla_{x, y, z} \mathbb{G}(t-s)\circledast(-\tfrac{1}{2}|\nabla_{\scriptscriptstyle{x^\prime, y^\prime, z^\prime}} \phi|^2(s))ds\Big\|_{L^2}.
\end{align*}

For the first term, we have
\begin{equation*}
\begin{split}
&\quad\Big\|\int^{\frac{t}{2}}_0\nabla_{\scriptscriptstyle{x, y, z}}\mathbb{G}(t-s)\circledast {(-\tfrac{1}{2}|\nabla_{\scriptscriptstyle{x^\prime, y^\prime, z^\prime}} \phi|^2(s))}ds\Big\|_{L^2}
\\[4pt]&\leq \int_0^{\frac{t}{2}}\interleave\nabla_{\scriptscriptstyle{x, y, z}}\mathbb{G}(t-s)\interleave_{L^{2}}\|\nabla_{\scriptscriptstyle{x^\prime, y^\prime, z^\prime}} \phi(s)\|^2_{L^{2}}ds
\\[4pt]&\leq Ce^{-\varepsilon t}\,\int_0^{\frac{t}{2}}(t-s)^{-\frac{3}{4}(1-\frac{1}{2})-\frac{1}{4}}(1+(A(t-s))^4)^{-\frac{1}{4}(1-\frac{1}{2})}s^{-\frac{2}{3}}(1+s)^{-\frac{19}{12}}d s
\\[4pt]&\leq C e^{-\varepsilon t}\,A^{-\frac{1}{2}}\int_0^{\frac{t}{2}}(t-s)^{-\frac{9}{8}}s^{-\frac{2}{3}}(1+s)^{-\frac{19}{12}}ds
 \leq Ce^{-\varepsilon t}\,t^{-\frac{9}{8}}.
\end{split}
\end{equation*}
Alternatively, we have
\begin{equation*}
\begin{split}
&\quad\Big\|\int^{\frac{t}{2}}_0\nabla_{\scriptscriptstyle{x, y, z}}\big(\mathbb{G}(t-s)\circledast(-\tfrac{1}{2}|\nabla_{\scriptscriptstyle{x^\prime,y^\prime,z^\prime}} \phi|^2(s))ds\big)\Big\|_{L^2}
\\[4pt]&\leq \int_0^{\frac{t}{2}}\interleave\nabla_{\scriptscriptstyle{x, y, z}}\mathbb{G}(t-s)\interleave_{L^{\frac{7}{6}}}\|\nabla_{\scriptscriptstyle{x^\prime, y^\prime, z^\prime}} \phi(s)\|^2_{L^{\frac{28}{9}}}ds
\\[4pt]&\leq Ce^{-\varepsilon t}\,\int_0^{\frac{t}{2}}(t-s)^{-\frac{3}{4}(1-\frac{6}{7})-\frac{1}{4}}(1+(A(t-s))^4)^{-\frac{1}{4}(1-\frac{6}{7})}s^{-\frac{2}{3}}ds
\\[4pt]&\leq Ce^{-\varepsilon t}\,\int_0^{\frac{t}{2}}(t-s)^{-\frac{7}{4}(1-\frac{6}{7})-\frac{1}{4}}s^{-\frac{2}{3}}ds\leq Ce^{-\varepsilon t}\,t^{-\frac{1}{6}}.
\end{split}
\end{equation*}

Therefore, we also have,
\begin{equation}\label{eq:ul2-six}
\begin{split}
\Big\|\int^{\frac{t}{2}}_0\nabla_{\scriptscriptstyle{x, y, z}}\mathbb{G}(t-s)\circledast{(-\tfrac{1}{2}|\nabla_{\scriptscriptstyle{x^\prime, y^\prime, z^\prime}} \phi|(s))}ds\Big\|_{L^2}\leq \delta e^{-\varepsilon t}\,t^{-\frac{1}{6}}(1+t)^{-\frac{23}{24}}.
\end{split}
\end{equation}

It remains to estimate the second term, for which we have
\begin{align}
&\quad\Big\|\int_{\frac{t}{2}}^t\nabla_{\scriptscriptstyle{x, y, z}}\big(\mathbb{G}(t-s)\circledast {(-\tfrac{1}{2}|\nabla_{\scriptscriptstyle{x^\prime, y^\prime, z^\prime}} \phi|^2(s))}ds\big)\Big\|_{L^2}\nonumber
\\[4pt]&\leq \int_{\frac{t}{2}}^t\interleave\nabla_{\scriptscriptstyle{x, y, z}}\mathbb{G}(t-s)\interleave_{L^{\frac{7}{6}}}\|\nabla_{\scriptscriptstyle{x^\prime, y^\prime, z^\prime}} \phi(s)\|^2_{L^{\frac{28}{9}}}ds\nonumber
\\[4pt]&\leq Ce^{-\varepsilon t}\,\int_{\frac{t}{2}}^t(t-s)^{-\frac{1}{2}}s^{-\frac{2}{3}}(1+s)^{-\frac{53}{24}}ds
\leq C e^{-\varepsilon t}\,t^{-\frac{1}{6}}(1+t)^{-\frac{23}{24}}.\label{eq:l222-2}
\end{align}
As a consequence of \eqref{eq:l222-1}, \eqref{eq:ul2-six}, and \eqref{eq:l222-2}, we obtain
\begin{equation*}
\begin{split}
\Big\|\int^{t}_0\nabla_{\scriptscriptstyle{x, y, z}}\mathbb{G}(t-s)\circledast {(-\tfrac{1}{2}|\nabla_{\scriptscriptstyle{x^\prime, y^\prime, z^\prime}} \phi|^2(s))}ds\Big\|_{L^2}\leq Ce^{-\varepsilon t}\,t^{-\frac{1}{6}}(1+t)^{-\frac{23}{24}}.
\end{split}
\end{equation*}

Therefore, we obtain the following bound for $\|\nabla \phi(t)\|_{L^2}$,
\begin{equation*}
\begin{split}
 \|\nabla \phi(t)\|_{L^2}\leq Ce^{-\varepsilon t}\,t^{-\frac{1}{6}}(1+t)^{-\frac{23}{24}}.
\end{split}
\end{equation*}
Similarly, we have
\begin{equation*}
\begin{split}
  \|\nabla \phi(t)\|_{L^\infty}\leq C e^{-\varepsilon t}\,t^{-\frac{1}{6}}(1+t)^{-\frac{11}{6}}.
\end{split}
\end{equation*}
This completes the proof of \eqref{eq:lem-deltau}. Under the assumptions of
\eqref{eq:lem-deltau}, the proof of Lemma~\ref{lem:regularity} follows by an
entirely similar argument.
\end{proof}

Based on this lemma, we show that
\begin{equation}\label{eq:regularity}
\|\nabla \phi(t)\|_{L^\infty}\le C e^{-\varepsilon t}\,(1+t)^{-2}
\qquad \text{and}\qquad
\|\nabla \phi(t)\|_{L^2}\le C e^{-\varepsilon t}\,(1+t)^{-\frac{9}{8}}
\end{equation}
constitute a regularity criterion for solutions to
\eqref{eq:mainequation}.
In what follows, we use this criterion to establish the global-in-time
existence of solutions to \eqref{eq:mainequation}.

According to a standard continuation argument, a solution to
\eqref{eq:mainequation} exists globally in time provided that its
$L^\infty$-norm remains bounded.
We therefore state the following lemma.
\begin{lemma}\label{lem:linfinity}
Assume that a solution $\phi(t,x,y,z)$ to \eqref{eq:mainequation}
satisfies the regularity criterion \eqref{eq:regularity}.
Then, for any $t>0$, it holds that
\begin{equation}\label{eq:infty}
\|\phi(t)\|_{L^\infty}\le Ce^{-\varepsilon t}\,(1+t)^{-\frac{7}{4}}.
\end{equation}
\end{lemma}

\begin{proof}
By the Green's function representation and Duhamel's principle, the solution
to \eqref{eq:mainequation} can be written as
\begin{equation*}
\phi(t,x,y,z)
= \mathbb{G}(t)\circledast \phi_0
+ \int_0^t \mathbb{G}(t-s)\circledast
\Big(-\tfrac12 |\nabla_{x',y',z'}\phi(s)|^2\Big)\,ds.
\end{equation*}
It follows that
\begin{equation*}
\|\phi(t)\|_{L^\infty}
\le \|\mathbb{G}(t)\circledast \phi_0\|_{L^\infty}
+ \Big\|\int_0^t \mathbb{G}(t-s)\circledast
\Big(-\tfrac12 |\nabla_{x',y',z'}\phi(s)|^2\Big)\,ds\Big\|_{L^\infty}.
\end{equation*}

By Theorem~\ref{thm:greenfunction}, we have
\begin{align*}
\|\mathbb{G}(t)\circledast \phi_0\|_{L^\infty}
&\le \|\mathbb{G}(t)\|_{L^\infty}\|\phi_0\|_{L^1}
\le C e^{-\varepsilon t} t^{-\frac34}(1+(At)^4)^{-\frac14}
\le C e^{-\varepsilon t} t^{-\frac74},
\\
\|\mathbb{G}(t)\circledast \phi_0\|_{L^\infty}
&\le \|\mathbb{G}(t)\|_{L^1}\|\phi_0\|_{L^\infty}
\le C.
\end{align*}
Combining the two estimates yields
\begin{equation}\label{eq:linfty-linear}
\|\mathbb{G}(t)\circledast \phi_0\|_{L^\infty}
\le C e^{-\varepsilon t}(1+t)^{-\frac74}.
\end{equation}

Next, we estimate the nonlinear term by splitting the time integral as
\begin{align*}
&\quad\Big\|\int_0^t \mathbb{G}(t-s)\circledast
\Big(-\tfrac12 |\nabla\phi(s)|^2\Big)\,ds\Big\|_{L^\infty}
\\
&\le
\Big\|\int_0^{t/2} \mathbb{G}(t-s)\circledast
\Big(-\tfrac12 |\nabla\phi(s)|^2\Big)\,ds\Big\|_{L^\infty}
+
\Big\|\int_{t/2}^t \mathbb{G}(t-s)\circledast
\Big(-\tfrac12 |\nabla\phi(s)|^2\Big)\,ds\Big\|_{L^\infty}
\\
&=: J_1 + J_2 .
\end{align*}

For $J_1$, using the $L^\infty$ bound of $\mathbb{G}$ and the $L^2$ decay of
$\nabla\phi$, we obtain
\begin{align*}
J_1
&\le \int_0^{t/2}
\|\mathbb{G}(t-s)\|_{L^\infty}\|\nabla\phi(s)\|_{L^2}^2\,ds
\\
&\le C e^{-\varepsilon t}
\int_0^{t/2}
(t-s)^{-\frac34}(1+(A(t-s))^4)^{-\frac14}(1+s)^{-\frac94}\,ds
\\
&\le C e^{-\varepsilon t} t^{-\frac74}.
\end{align*}
On the other hand, we have
\begin{align*}
J_1
&\le \int_0^{t/2}
\|\mathbb{G}(t-s)\|_{L^1}\|\nabla\phi(s)\|_{L^\infty}^2\,ds
\\
&\le C e^{-\varepsilon t}
\int_0^{t/2} s^{-\frac23}(1+s)^{-\frac{10}{3}}\,ds
\le C.
\end{align*}
Combining the above bounds, we conclude that
\begin{equation}\label{eq:J1}
J_1 \le C e^{-\varepsilon t}(1+t)^{-\frac74}.
\end{equation}

For $J_2$, one has
\begin{align*}
J_2
&\le \int_{t/2}^t
\|\mathbb{G}(t-s)\|_{L^2}\|\nabla\phi(s)\|_{L^4}^2\,ds
\\
&\le C e^{-\varepsilon t}
\int_{t/2}^t
(t-s)^{-\frac38}(1+(A(t-s))^4)^{-\frac18}
(1+s)^{-2}(1+s)^{-\frac98}\,ds
\\
&\le C e^{-\varepsilon t} t^{\frac18}(1+t)^{-2}
\le C e^{-\varepsilon t}(1+t)^{-\frac74}.
\end{align*}

Combining \eqref{eq:linfty-linear}, \eqref{eq:J1}, and the above estimate for $J_2$, we obtain
\begin{equation*}
\|\phi(t)\|_{L^\infty}
\le C e^{-\varepsilon t}(1+t)^{-\frac74}.
\end{equation*}
This completes the proof of Lemma~\ref{lem:linfinity}.
\end{proof}

Consequently, the local-in-time existence of solutions to
equation~\eqref{eq:mainequation} extends to global-in-time existence.
We have the following Theorem,
\begin{thm}\label{thm:global}
Let $\phi_0(x,y,z) \in W^{4,\infty}(\mathbb{R}^3)\cap L^1(\mathbb{R}^3)$, and let $\varepsilon>0$ be a fixed constant with $\kappa>\frac{1}{16}+\varepsilon$. Then there exists a positive constant $A_0 = A_0(\phi_0, \kappa)$ such that, for any $A \ge A_0$, the Cauchy problem \eqref{eq:mainequation} admits a unique global classical solution
\begin{align}\label{eq:globalint}
  \phi(t,x,y,z) \in C\big([0,\infty);\, W^{4,\infty}(\mathbb{R}^3)\cap L^1(\mathbb{R}^3)\big).
\end{align}
The constant $C>0$ depends only on the initial data $\phi_0$ and is independent of $t$.
\end{thm}

\section{Decay Estimates of Solution}
Finally, we establish decay estimates for higher-order derivatives $D^k \phi(t, x, y, z)$. We have the following lemma,
\begin{lemma}\label{lem:decayinfty}
Suppose that $\phi(t,x,y,z)$ is a global classical solution to
\eqref{eq:mainequation} constructed in Theorem~\ref{thm:global}.
Then for $0\le k\le 4$, the following estimates hold:
\[
\|D^k\phi(t)\|_{L^\infty}
\le
Ce^{-\varepsilon t}(1+t)^{-\frac{3}{4}-\frac{k}{4}}
(1+(At)^4)^{-\frac{1}{4}}.
\]
\end{lemma}
\begin{proof}
Similarly to the proofs of Lemma~\ref{lem:regularity} and Lemma~\ref{lem:linfinity}, we easily obtain the following estimates:
\begin{align}
\|\phi(t)\|_{L^\infty} &\le C e^{-\varepsilon t}\,(1+t)^{-\frac{3}{4}} \left(1+(At)^4\right)^{-\frac{1}{4}}, \label{eq:phi-linfty} \\[5pt]
\|\nabla \phi(t)\|_{L^\infty} &\le Ce^{-\varepsilon t}\, (1+t)^{-1} \left(1+(At)^4\right)^{-\frac{1}{4}}. \label{eq:grad-phi-linfty}
\end{align}

We now turn to prove the decay estimate of $\|D^2 \phi(t)\|_{L^\infty}$. By Duhamel's Principle, we have
\begin{equation*}
\begin{split}
 &\quad\left\|D^2\phi(t)\right\|_{L^\infty}=\left \|D_{\scriptscriptstyle{x, y, z}}^2\big(\mathbb{G}(t)\circledast \phi_0\big)\right\|_{L^\infty}+\Big\|\int^t_0 D_{\scriptscriptstyle{x, y, z}}^2\big(\mathbb{G}(t-s)\circledast {(-\tfrac{1}{2}|\nabla_{\scriptscriptstyle{x^\prime, y^\prime, z^\prime}} \phi|^2)}(s)\big)ds\Big\|_{L^\infty}.
\end{split}
\end{equation*}
For the linear part, by the decay estimates of the Green's function in
Theorem~\ref{thm:greenfunction}, we have
\begin{equation*}
\begin{split}
\big \|D_{\scriptscriptstyle{x, y, z}}^2\big(\mathbb{G}(t)\circledast \phi_0\big)\big\|_{L^\infty}
  \leq \interleave D_{\scriptscriptstyle{x, y, z}}^2 \mathbb{G}(t)\interleave_{L^\infty}\|\phi_0\|_{L^1}
\leq Ce^{-\varepsilon t}\,t^{-\frac{5}{4}}(1+(At)^4)^{-\frac{1}{4}}.
\end{split}
\end{equation*}

From the above inequality and $\eqref{eq:globalint}$,  we have
\begin{equation}\label{eq:d2-0}
\big \|D_{\scriptscriptstyle{x, y, z}}^2\big(\mathbb{G}(t)\circledast \phi_0\big)\big\|_{L^\infty}\leq Ce^{-\varepsilon t}\,(1+t)^{-\frac{5}{4}}(1+(At)^4)^{-\frac{1}{4}}.
\end{equation}

For the nonlinear part, as before, we decompose it into two parts and obtain
{\begin{align*}
  &\quad\Big\|\int^t_0 D_{\scriptscriptstyle{x, y, z}}^2\big(\mathbb{G}(t-s)\circledast {(-\tfrac{1}{2}|\nabla_{\scriptscriptstyle{x^\prime, y^\prime, z^\prime}} \phi|^2)}(s)\big)ds\Big\|_{L^\infty}
  \\[6pt]&\leq \Big\|\int^{\frac{t}{2}}_0D_{\scriptscriptstyle{x, y, z}}^2\big(\mathbb{G}(t-s)\circledast (-\tfrac{1}{2}|\nabla_{\scriptscriptstyle{x^\prime, y^\prime, z^\prime}} \phi|^2)(s)\big)ds\Big\|_{L^\infty}
  \\[6pt]&\quad+\Big\|\int^t_{\frac{t}{2}}D_{\scriptscriptstyle{x, y, z}}^2\big(\mathbb{G}(t-s)\circledast (-\tfrac{1}{2}|\nabla_{\scriptscriptstyle{x^\prime, y^\prime, z^\prime}} \phi|^2)(s)\big)ds\Big\|_{L^\infty}
  \\[6pt]&=:N_1+N_2.
\end{align*}}
For the estimate of $N_1$, we have
\begin{equation*}
   \begin{split}
    N_1&\leq \int^{\frac{t}{2}}_0\interleave D_{\scriptscriptstyle{x, y, z}}^2\mathbb{G}(t-s)\interleave _{L^\infty}\|\nabla_{\scriptscriptstyle{x^\prime, y^\prime, z^\prime}} \phi(s)\|_{L^2}^2ds
\\[4pt]&\leq Ce^{-\varepsilon t}\,\int_0^{\frac{t}{2}}(t-s)^{-\frac{3}{4}-\frac{2}{4}}(1+(A(t-s))^4)^{-\frac{1}{4}}(1+s)^{-\frac{9}{4}}ds
\\[4pt]&\leq Ce^{-\varepsilon t}\,t^{-\frac{5}{4}}(1+(At)^4)^{-\frac{1}{4}}.
   \end{split}
\end{equation*}
Alternatively, using the $L^1$--$L^\infty$ estimate for the Green's function, we also have
\begin{equation*}
   \begin{split}
    N_1&\leq \int_0^{\frac{t}{2}}\interleave D_{\scriptscriptstyle{x, y, z}}^2\mathbb{G}(t-s)\interleave _{L^1}\|\nabla_{\scriptscriptstyle{x^\prime, y^\prime, z^\prime}} \phi(s)\|^2_{L^\infty}ds
    \\[4pt]&\leq Ce^{-\varepsilon t}\,\int_0^{\frac{t}{2}}(t-s)^{-\frac{1}{2}}(1+s)^{-4}ds
\leq C.
   \end{split}
\end{equation*}
Therefore, we have
\begin{equation}\label{eq:d2-1}
  N_1=\Big\|\int^{\frac{t}{2}}_0 D_{\scriptscriptstyle{x, y, z}}^2\big((\mathbb{G}(t-s)\circledast(-\tfrac{1}{2}|\nabla_{\scriptscriptstyle{x^\prime, y^\prime, z^\prime}} \phi|^2)(s)\big)ds\Big\|_{L^\infty}\leq Ce^{-\varepsilon t}\,(1+t)^{-\frac{5}{4}}(1+(At)^4)^{-\frac{1}{4}}.
\end{equation}
For the estimate of ${N}_2$, we have
\begin{align}
  N_2&\leq\int_{\frac{t}{2}}^t\interleave D_{\scriptscriptstyle{x, y, z}}^2\mathbb{G}(t-s)\interleave_{L^1}\|\nabla_{\scriptscriptstyle{x^\prime, y^\prime, z^\prime}} \phi(s)\|_{L^\infty}^2ds\nonumber
\\[4pt]& \leq Ce^{-\varepsilon t}\,\int_{\frac{t}{2}}^t(t-s)^{-\frac{2}{4}}(1+s)^{-4}ds\leq C(1+t)^{-\frac{9}{4}}\nonumber
\\[4pt]&\leq Ce^{-\varepsilon t}\,(1+t)^{-\frac{5}{4}}(1+(At)^4)^{-\frac{1}{4}}.\label{eq:d2-2}
\end{align}
Combining \eqref{eq:d2-0}-\eqref{eq:d2-2}, we obtain
\begin{equation*}
  \|D^2\phi(t)\|_{L^\infty}\leq Ce^{-\varepsilon t}\,(1+t)^{-\frac{5}{4}}(1+(At)^4)^{-\frac{1}{4}}.
\end{equation*}
The estimate for $\|D^3 \phi\|_{L^\infty}$ is similar to $\|D^2 \phi\|_{L^\infty}$ and omitted for brevity:
\[
\|D^3 \phi(t)\|_{L^\infty} \le C e^{-\varepsilon t} (1+t)^{-\frac{3}{2}} (1+(At)^4)^{-\frac{1}{4}}.
\]

For the decay estimate of $\|D^4 \phi(t)\|_{L^\infty}$, due to the singularity of the Green's function, a more delicate estimate is required.
\begin{equation*}
\begin{split}
 \big\|D^4\phi(t)\big\|_{L^\infty}&=\big \|D^4_{\scriptscriptstyle{x, y, z}}\big(\mathbb{G}(t)\circledast \phi_0\big)\big\|_{L^\infty}+\Big\|\int^t_0D^4_{\scriptscriptstyle{x, y, z}}\big(\mathbb{G}(t-s)\circledast (-\tfrac{1}{2}|\nabla_{\scriptscriptstyle{x^\prime, y^\prime, z^\prime}} \phi|^2)(s)\big)ds\Big\|_{L^\infty}.
\end{split}
\end{equation*}

For the linear part, by Theorem $\ref{thm:greenfunction}$, we have
\begin{align}\label{eq:d4-0}
\big \|D^4_{\scriptscriptstyle{x, y, z}}\big(\mathbb{G}\circledast \phi_0\big)(t)\big\|_{L^\infty}\leq Ce^{-\varepsilon t}\,(1+t)^{-\frac{7}{4}}(1+(At)^4)^{-\frac{1}{4}}.
\end{align}
For the nonlinear part, we also divided it into two parts,
\begin{equation*}
  \begin{split}
&\quad\Big\|\int^t_0 D^4_{\scriptscriptstyle{x, y, z}}\big(\mathbb{G}(t-s)\circledast (-\tfrac{1}{2}|\nabla_{\scriptscriptstyle{x^\prime, y^\prime, z^\prime}} \phi|^2)(s)\big)ds\Big\|_{L^\infty}
 \\[6pt]&=\Big\|\int^{\frac{t}{2}}_0D^4_{\scriptscriptstyle{x, y, z}}\big(\mathbb{G}(t-s)\circledast (-\tfrac{1}{2}|\nabla_{\scriptscriptstyle{x^\prime, y^\prime, z^\prime}} \phi|^2)(s)\big)ds\Big\|_{L^\infty}
  \\[6pt]&\quad+\Big\|\int^t_{\frac{t}{2}}D^4_{\scriptscriptstyle{x, y, z}}\big(\mathbb{G}(t-s)\circledast (-\tfrac{1}{2}|\nabla_{\scriptscriptstyle{x^\prime, y^\prime, z^\prime}} \phi|^2)(s)\big)ds\Big\|_{L^\infty}
 \\[6pt]&=:{L}_1+L_2.
  \end{split}
\end{equation*}
For the nonlinear part $L_1$, we have
\begin{equation*}
   \begin{split}
    L_1&\leq \int^{\frac{t}{2}}_0\interleave D^4_{\scriptscriptstyle{x, y, z}}\mathbb{G}(t-s)\interleave _{L^\infty}\|\nabla_{\scriptscriptstyle{x^\prime, y^\prime, z^\prime}} \phi(s)\|_{L^2}^2ds
\\[4pt]&\leq Ce^{-\varepsilon t}\,\int_0^{\frac{t}{2}}(t-s)^{-\frac{3}{4}-\frac{4}{4}}(1+(A(t-s))^4)^{-\frac{1}{4}}(1+s)^{-\frac{9}{4}}ds
\\[4pt]&\leq Ce^{-\varepsilon t}\,t^{-\frac{7}{4}}(1+(At)^4)^{-\frac{1}{4}}.
   \end{split}
\end{equation*}
For the other hand, we have
\begin{equation*}
   \begin{split}
    L_1&\leq \int^{\frac{t}{2}}_0\interleave D^3_{\scriptscriptstyle{x, y, z}}\mathbb{G}(t-s)\interleave _{L^1}(1+At)\|D^2_{\scriptscriptstyle{x^\prime, y^\prime, z^\prime}}\phi(s)\|_{L^\infty}\|\nabla_{\scriptscriptstyle{x^\prime, y^\prime, z^\prime}} \phi(s)\|_{L^\infty}ds
    \\[4pt]&\leq Ce^{-\varepsilon t}\,(1+At)\int_0^{\frac{t}{2}}(t-s)^{-\frac{3}{4}}(1+s)^{-\frac{17}{4}}ds
\leq Ce^{-\varepsilon t}\,t^{\frac{1}{4}} (1+t),
   \end{split}
\end{equation*}
where we have applied the correspondence between the differential operators $D_{x,y,z}$ and $D_{x',y',z'}$ as described in \eqref{eq:transfer}.
Thus, we have
\begin{align}
  L_1&=\Big\|\int^{\frac{t}{2}}_0 D^4_{\scriptscriptstyle{x, y, z}}\big((\mathbb{G}(t-s)\circledast(-\tfrac{1}{2}|\nabla_{\scriptscriptstyle{x^\prime, y^\prime, z^\prime}} \phi|^2)(s)\big)ds\Big\|_{L^\infty}\nonumber
  \\[4pt]&\leq Ce^{-\varepsilon t}\,(1+t)^{-\frac{7}{4}}(1+(At)^4)^{-\frac{1}{4}}.\label{eq:d4-1}
\end{align}
For the estimates of ${L}_2$, we have
\begin{align}
  L_2&\leq C\int_{\frac{t}{2}}^t\interleave D^3_{\scriptscriptstyle{x, y, z}}\mathbb{G}(t-s)\interleave _{L^1}(1+At)\|D^2_{\scriptscriptstyle{x^\prime, y^\prime, z^\prime}}\phi(s)\|_{L^\infty}\|\nabla_{\scriptscriptstyle{x^\prime, y^\prime, z^\prime}}\phi(s)\|_{L^\infty}ds\nonumber
\\[4pt]& \leq Ce^{-\varepsilon t}\,\int_{\frac{t}{2}}^t(t-s)^{-\frac{3}{4}}(1+s)^{-\frac{17}{4}}ds
\leq Ct^{\frac{1}{4}}(1+t)^{-\frac{17}{4}}\leq C(1+t)^{-\frac{11}{4}}\nonumber
\\[4pt]&\leq Ce^{-\varepsilon t}\,(1+t)^{-\frac{7}{4}}(1+(At)^4)^{-\frac{1}{4}}.\label{eq:d4-2}
\end{align}
By \eqref{eq:d4-0}, \eqref{eq:d4-1} and \eqref{eq:d4-2}, we have
\begin{equation*}
  \|D^4\phi(t)\|_{L^\infty}\leq Ce^{-\varepsilon t}\,(1+t)^{-\frac{7}{4}}(1+(At)^4)^{-\frac{1}{4}}.
\end{equation*}

Consequently, we obtain the decay estimates of $\|D^k \phi(t)\|_{L^\infty}$,
\begin{equation}\label{eq:d1-4}
 \begin{split}
    \|D^k \phi(t)\|_{L^\infty}\leq Ce^{-\varepsilon t}\,(1+t)^{-\frac{3}{4}-\frac{k}{4}}(1+(At)^4)^{-\frac{1}{4}},  \qquad 0\leq k\leq 4.
\end{split}
\end{equation}
Thus, we finish the proof of Lemma $\ref{lem:decayinfty}$.
\end{proof}

By combining the $L^\infty$ estimates for $D^k \phi$ with interpolation and the Gagliardo--Nirenberg inequality, we obtain time-decay estimates of the solution to \eqref{eq:mainequation}.
\begin{thm}[Decay Estimates]\label{thm:decay}
Under the same assumptions as in Theorem~\ref{thm:global}, the global solution $\phi(t,x,y,z)$ satisfies the following decay estimates:
\begin{equation}\label{eq:decay1-3}
\| D^k \phi(t) \|_{L^p}
\le C\, e^{-\varepsilon t} \,(1+t)^{-\frac{3}{4}(1-\frac{1}{p})-\frac{k}{4}} \,(1+(At)^4)^{-\frac{1}{4}(1-\frac{1}{p})},
\qquad 0 \le k < 4,\ \ p \ge 1,
\end{equation}
where $D^k = \partial_x^{k_1}\partial_y^{k_2}\partial_z^{k_3}$ denotes any multi-index derivative with $k_1+k_2+k_3=k$, and the constant $C>0$ depends only on the initial data $\phi_0$ and is independent of $t$.
\end{thm}

\begin{proof}
By the interpolation inequality and $\eqref{eq:infty}$, for any $p\geq1$,
\begin{equation*}
  \|\phi(t)\|_{L^p}\leq Ce^{-\varepsilon t}\,\|\phi(t)\|_{L^1}^{\frac{1}{p}}\ \|\phi(t)\|_{L^\infty}^{1-\frac{1}{p}}\leq Ce^{-\varepsilon t}\,{(1+t)}^{-\frac{3}{4}(1-\frac{1}{p})}(1+(At)^4)^{-\frac{1}{4}(1-\frac{1}{p})}.
\end{equation*}
Then by Gagliardo-Nirenberg inequality, we have
\begin{equation*}
  \|D^k \phi(t)\|_{L^p}\leq C\,\|D^4 \phi(t)\|_{L^\infty}^{\theta}\cdot\|\phi(t)\|_{L^p}^{1-\theta},
\end{equation*}
where $\theta\in(0,1)$ is determined by
\[
\frac{1}{p} = \frac{k}{3} - \frac{4\theta}{3} + \frac{1-\theta}{p},
\]
that is, $\theta=\dfrac{kp}{4p+3}$.

Therefore,
\begin{equation*}
  \|D^k \phi(t)\|_{L^p}\leq Ce^{-\varepsilon t}\,(1+t)^{-\frac{3}{4}(1-\frac{1}{p})-\frac{k}{4}}(1+(At)^4)^{-\frac{1}{4}},\ \ 0\leq k<4.
\end{equation*}

Combining the above estimates for all derivatives $0 \le k < 4$, we have established the desired decay bounds. This completes the proof of the decay estimate of $\|D^k \phi(t)\|_{L^p}$ for $0\le k<4$, and hence the proof of Theorem~\ref{thm:decay} is complete.
\end{proof}
\section*{Acknowledgments}
The authors sincerely thank the Institute of Science Tokyo for its generous support.
The research of the first author was partially supported by JSPS KAKENHI Grant-in-Aid for Scientific Research (A) (Grant No. 24H00185),
and the second author acknowledges support from the National Natural Science Foundation of China (Grant No. 12271357).
\section*{Data availability}
 No data was used for the research described in the article.

\end{document}